\documentclass{article}
\usepackage{amsmath,amsthm,amsfonts,amssymb,xcolor,enumerate,stmaryrd,mathtools}
\usepackage[all]{xy}

\newcommand{\CAlg}{\mathbf{CAlg}}
\newcommand{\C}{\mathbb{C}}

\newcommand{\Aut}{\textnormal{Aut}}
\newcommand{\Der}{\textnormal{Der}}

\newcommand{\SplExt}{\textbf{SplExt}}
\newcommand{\bwnh}{\textbf{(WNH)}}
\newcommand{\bnh}{\textbf{(NH)}}
\newcommand{\wnh}{\textnormal{(WNH)}}
\newcommand{\nh}{\textnormal{(NH)}}
\newcommand{\uce}{\textnormal{(UCE)}}
\newcommand{\sh}{\textnormal{(SH)}}
\newcommand{\comm}[2]{\gamma_{[#1, #2]}}
\newcommand{\como}[2]{[#1, #2]}
\newcommand{\spexc}[1]{\llbracket#1\rrbracket}

\newcommand{\ke}{\textnormal{ker}}
\newcommand{\coke}{\textnormal{coker}}

\newtheorem{theorem}{Theorem}[section]
\newtheorem{proposition}[theorem]{Proposition}
\newtheorem{corollary}[theorem]{Corollary}
\newtheorem{lemma}[theorem]{Lemma}
\newtheorem{example}[theorem]{Example}
\newtheorem{remark}[theorem]{Remark}

\begin{document}
\title{A note on split extension classifiers of perfect objects}
\author{J. R. A. Gray}
\maketitle
\begin{abstract}
It is well known that the automorphism group of a perfect group
has trivial center. This is a consequence of the fact that for
a perfect group
the centralizer of the normal subgroup of inner automorphisms of
its automorphism group is trivial.
We find a categorical explanation for this fact by showing that for
a semi-abelian category $\C$ satisfying \wnh{}, if $X$ is a perfect object in $\C$ such
that
the split extension classifier $\spexc{X}$ exists, then the centralizer of
the \emph{conjugation} morphism
$c_X : X\to \spexc{X}$ is trivial and hence the center of $\spexc{X}$ is trivial.
Our proof depends on the following fact, which we show holds in an arbitrary
semi-abelian category satisfying \wnh{}. 
If $f:X\to Y$ is a central extension and $\theta_1$, $\theta_2$ and $\phi$
are internal object actions of $B$ on $X$, $X$ and $Y$, respectively, 
such that $f:(X,\theta_i)\to (Y,\phi)$ is an internal $B$-action morphism,
then $\theta_1$ and $\theta_2$ restrict to the same action on $[X,X]$.
Using this latter fact we also prove that \uce{}
is a consequence of \wnh{} alone (rather than of \wnh{} together with \sh{} as previously
known).
\end{abstract}
\section{Introduction and preliminaries}
Recall that the automorphisms of a group that commute with every inner automorphism
are called central automorphisms. For a group $X$ these automorphisms form a normal subgroup
$\Aut_c(X)$ of the automorphism group $\Aut(X)$ which has been the subject of
several papers whose study goes back at least to J.~E.~Adney and T.~Yen
\cite{ADNEY_YEN:1965}.
One of the methods of studying these automorphisms is via the so-called
\emph{Adney-Yen}
map, which is an injective map from $\Aut_c(X)$ to the set
of homomorphisms from $X$ to its center $Z(X)$.
As an immediate consequence of the existence of this injective map
one sees
that if $X$ is perfect, then $\Aut_c(X)$ is trivial.

In this paper, we generalize this result to 
arbitrary \emph{semi-abelian action representable} categories
satisfying a mild condition on commutators (Corollary \ref{corollary:perf_action_rep}). 
We will see that this fact depends on a general fact about
central extensions which holds for all semi-abelian categories
satisfying the same mild condition on commutators (Proposition \ref{proposition:actions}). 
For Lie algebras, Corollary \ref{corollary:perf_action_rep}
recovers the following known result.
For a fixed commutative ring $R$
and a Lie algebra $X$ 
over $R$, the Lie algebra of derivations of $X$, $\Der(X)$, has underlying abelian group
the subgroup of $R$-linear maps $d:X\to X$ satisfying $d([x,y])=[d(x),y]+[x,d(y)]$
for each $x,y$ in $X$, and has Lie bracket defined by $[d,e]=d\circ e-e\circ d$.
For each $x$ in $X$, the map $y\mapsto [x,y]$ determines a derivation called
an inner derivation.
In analogy with the group-theoretic case (in fact, both
are special cases of the same categorical construction) derivations which
commute with (annihilate)
every inner derivation are called central and form an ideal of $\Der(X)$ which
we denote by $\Der_c(X)$. The analogous fact is that
if $X$ is perfect, then $\Der_c(X)$ is trivial.
For groups Proposition \ref{proposition:actions_2}, which is an easy corollary of the aforementioned Proposition \ref{proposition:actions}, 
recovers the fact that if $f:X\to Y$ is a central extension of groups, 
and $\cdot$ and $*$ are actions of a group $B$ on $X$ such that, for each $b$ in $B$ 
and each $x$ in $X$ we have $f(b\cdot x)=f(b*x)$, then both actions
have the same restriction to $[X,X]$. 

We now recall the necessary background and introduce notation
so that we can precisely state our main objectives.

For a pointed category $\C$, we will denote by $0$ a chosen
zero object (that is, an object which is both initial and terminal). For
objects $A$ and $B$ in $\C$ we will also denote by $0$
the unique morphism from $A$ to $B$ which factors through the zero
object $0$. The product of a pair of objects $A$ and $B$ will be written
$A\times B$ with first and second projection denoted by $\pi_1$ and $\pi_2$,
respectively. The coproduct will be written $A+B$ with coproduct inclusions
$\iota_1$ and $\iota_2$. For morphisms $u : W \to A$, $v: W\to B$, $f:A\to C$
and $g : B\to C$
we denote by $\langle u,v\rangle : W\to A\times B$ and
$[f,g]:A+B\to C$ the unique morphisms such that
$\pi_1 \langle u, v\rangle=u$, $\pi_2 \langle u,v\rangle = v$, $[f,g]\iota_1=f$
and $[f,g]\iota_2 = g$.

Recall that a split extension in a pointed category $\C$ is a
diagram
\begin{equation}
\label{diagram:split}
\vcenter{
\xymatrix{
X \ar[r]^{\kappa} & A\ar@<0.5ex>[r]^{\alpha} & B\ar@<0.5ex>[l]^{\beta}
}
}
\end{equation}
in $\C$ where $\alpha \beta = 1_B$ and $\kappa$ is the kernel of
$\alpha$. A morphism of split extensions in $\C$ is a diagram
\begin{equation}
\label{diagram:morphism_of_split}
\vcenter{
\xymatrix{
X \ar[r]^{\kappa}\ar[d]^{f} & A\ar@<0.5ex>[r]^{\alpha}\ar[d]^{g} &
B\ar@<0.5ex>[l]^{\beta}\ar[d]^{h}\\
Z\ar[r]_{\sigma} & C\ar@<0.5ex>[r]^{\gamma} & D\ar@<0.5ex>[l]^{\delta}
}
}
\end{equation}
such that the top and bottom rows are split extensions (the domain
and codomain, respectively) and $g\kappa = \sigma f$, $g\beta = \delta h$
and $h\alpha = \gamma g$. We denote by $\SplExt(\C)$ the category of
split extensions in $\C$, and by $K$ and $P$ the functors sending \eqref{diagram:split}
to $X$ and $B$ respectively,
and \eqref{diagram:morphism_of_split} to $f$ and $h$ respectively.
We denote by $K^{-1}(X)$ the fiber of the functor $K$ above $X$, that is
the subcategory of $\SplExt(\C)$ with objects and morphisms mapped by $K$
to $X$ and $1_X$, respectively. Following F.\ Borceux, G.\ Janelidze and
G.\ M.\ Kelly
\cite{BORCEUX_JANELIDZE_KELLY:2005a},
 we call a terminal object in $K^{-1}(X)$ a generic split
extension with kernel $X$ which we will denote by
\begin{equation}
\label{diagram:generic_split}
\vcenter{
\xymatrix{
 X \ar[r]^-{k} & \spexc{X}\ltimes X \ar@<0.5ex>[r]^-{p_1} & \spexc{X}.\ar@<0.5ex>[l]^-{i}
}
}
\end{equation}
If $X$ is an object in a pointed category $\C$ admitting a generic
split extension with kernel $X$, then we will denote by $c_X : X\to \spexc{X}$
the morphism forming part of the unique morphism
\begin{equation}
\label{diagram:conj}
\vcenter{
\xymatrix{
X \ar[r]^-{\langle 0,1\rangle} \ar@{=}[d]
& X\times X \ar@<0.5ex>[r]^-{\pi_1}\ar[d]^{\phi}
&  X \ar@<0.5ex>[l]^-{\langle 1,1\rangle}\ar[d]^{c_X}\\
X \ar[r]^-{k} & \spexc{X}\ltimes X \ar@<0.5ex>[r]^-{p_1} & \spexc{X}\ar@<0.5ex>[l]^-{i}
}
}
\end{equation}
in $K^{-1}(X)$. Recall that there is also a morphism
$p_2 : \spexc{X}\ltimes X\to \spexc{X}$ which, together with $p_1$ and $i$,
forms part of
a canonical groupoid structure (discovered in
\cite{BORCEUX_BOURN:2007}, see also the preliminaries of
\cite{GRAY:2022}),
and furthermore $c_X = p_2 k$.
Note that if $\C$ is the category of groups, then $\spexc{X}=\Aut(X)$
and $c_X$ is the homomorphism sending each $x$ in $X$ to the inner
automorphism determined by conjugation 
(that is $x \mapsto (y \mapsto xyx^{-1})$) \cite{BORCEUX_JANELIDZE_KELLY:2005b}.
On the other hand, if $\C$ is the category of Lie algebras over a
commutative ring, then $\spexc{X}=\Der(X)$ and $c_X$ is the homomorphism
sending each $x$ in $X$ to the inner derivation determined by $x$ (that is
$x\mapsto (y\mapsto [x,y])$) \cite{BORCEUX_JANELIDZE_KELLY:2005b}.
Recall also that a split
extension \eqref{diagram:split} in a category $\C$ is called faithful if there is at most
one morphism from any split extension in $K^{-1}(X)$ to it,
and the category $\C$ is called action accessible \cite{BOURN_JANELIDZE:2009} if every split extension
admits a morphism into a faithful one in the same fiber of $K$.

Recall that a pointed finitely complete category can equivalently be defined
to be (Bourn)-protomodular
\cite{BOURN:1991}
if it satisfies the split short five lemma, that is, for
each morphism of split extensions \eqref{diagram:morphism_of_split}, if $f$ and
$h$ are isomorphisms, then $g$ is an isomorphism. A category is regular if
it has all finite limits, regular epimorphisms are pullback stable in it, and
coequalizers of kernel pairs exist in it. A category is homological 
if it is a regular pointed protomodular category \cite{BORCEUX_BOURN:2004}.

A pair of
morphisms $f:A \to C$ and $g: B\to C$ in a pointed protomodular category
$\C$ with binary products (Huq)-commute \cite{HUQ:1968} if there exists a morphism
$\varphi : A\times B \to C$ (called the cooperator) making the diagram
\[
\xymatrix{
A\ar[r]^-{\langle 1,0\rangle}\ar[dr]_{f} &
A\times B \ar[d]^{\varphi}
& B\ar[l]_-{\langle 0,1\rangle}\ar[dl]^{g}\\
& C &
}
\]
commute. For $f$ and $g$ as above, we use the following (equivalent) definition
of the Huq-commutator (originally defined in \cite{HUQ:1968}).
The Huq commutator
$\comm{f}{g} : \como{(A,f)}{(B,g)} \to C$ of $f$ and $g$ is
initial amongst those normal monomorphisms $n$ with codomain $C$ such that
 $\coke(n)f$ and $\coke(n)g$ commute (where $\coke(n)$ is the cokernel of $n$
and by normal monomorphism we mean a morphism which is the kernel of some
morphism). Note that we will write $\como{X}{X}$ instead of
$\como{(X,1_X)}{(X,1_X)}$.
The centralizer of a morphism $f: A\to C$ is the terminal object in
the category of morphisms which commute with $f$ and will be denoted
$z_f:Z_C(A,f)\to C$. The center of $X$ is the centralizer of $1_X$ and
will be written $z_X:Z(X)\to X$.
Recall also that a morphism $f:A\to X$ is called central if it commutes with
$1_X$, and that an object
$X$ is commutative if $1_X$ commutes with $1_X$ (which in the protomodular
context implies that it is abelian, that is, it admits a unique internal
abelian group structure). An object is perfect if the only
morphisms from it to commutative objects are zero morphisms. This implies
that for a perfect object $X$, the commutator morphism $\comm{1_X}{1_X} : \comm{X}{X}\to X$
exists and is an isomorphism.

Recall that a category $\C$ is (Barr)-exact
\cite{BARR:1971},
if it is regular, and equivalence relations
in $\C$ are effective (that is, they are the kernel pair of some morphism).
A pointed protomodular
category is semi-abelian
\cite{JANELIDZE_MARKI_THOLEN:2002}
if it has binary coproducts and is Barr-exact.
In agreement with the definitions in the semi-abelian context
(see \cite{GRAY_VAN_DER_LINDEN:2015} and \cite{CIGOLI_GRAY_VAN_DER_LINDEN:2015a},
respectively) we consider the following conditions on a pointed category $\C$:\\

\noindent\bwnh{}. For each normal monomorphism
$k : K\to X$ in $\C$, if the commutator $\comm{1_K}{1_K}$ exists,
then the morphism $k\comm{1_K}{1_K}$ is the commutator morphism of $k$ with itself.\\

\noindent\bnh{}. For monomorphisms $m:X\to Y$ and $k:K\to X$ and $l:L\to X$ in $\C$
if $\comm{k}{l}$ exists and $mk$ and $ml$ are normal monomorphisms,
then $m\comm{k}{l}$ is the commutator morphism of $mk$ and $ml$.\\

It turns out that \nh{} and \wnh{} can be
reformulated in terms of Higgins commutators in the appropriate context where those
commutators exist (see \cite{GRAY_VAN_DER_LINDEN:2015} and \cite{CIGOLI_GRAY_VAN_DER_LINDEN:2015a} where, in fact,
\wnh{} and \nh{} were defined in terms of Higgins commutators).

For a semi-abelian category $\C$ and an object $B$ in $\C$, let us
recall from \cite{BOURN_JANELIDZE:1998} that
the functor sending a split extension in $P^{-1}(B)$ (the fiber of
$P:\SplExt(\C)\to \C$ above $B$) to its kernel
is monadic, and the algebras of this monad, whose underlying
functor we denote by $B\flat -$, are called internal $B$-actions.
The induced equivalence associates to each split extension
\eqref{diagram:split}
the pair $(X,\chi)$, where $\chi$ is
the unique morphism making the diagram
\[
\xymatrix{
B\flat X
\ar[r]^{k}
\ar[d]_{\chi}
&
B+X
\ar@<0.5ex>[r]^-{[1,0]}
\ar[d]_{[\beta,\kappa]}
&
B
\ar@<0.5ex>[l]^-{\iota_1}
\ar@{=}[d]
\\
X
\ar[r]^{\kappa}
&
A\ar@<0.5ex>[r]^{\alpha}
&
B,
\ar@<0.5ex>[l]^{\beta}
\\
}
\]
in which $k$ is the kernel of $[1,0]$, commutes. When $\C$
is the category of groups,
it turns out that internal $B$-actions
are in one to one correspondence with the usual $B$-actions. 

Recall that a normal monomorphism $k:K\to X$ in a pointed category $\C$
is characteristic (in
agreement with the definition in \cite{CIGOLI_MONTOLI:2015}) if for
each normal monomorphism $n:X\to Y$ the composite $nk$ is a normal monomorphism.
Note that using this language \wnh{} is equivalent to requiring for each $X$,
if $\comm{1_X}{1_X}$ exists, then it is characteristic.

Recall also that a regular epimorphism $f:A\to B$ is a central extension if its kernel is central (i.e. Huq-commutes with $1_A$).
A semi-abelian category satisfies \uce{}, introduced and studied in 
\cite{CASAS_VAN_DER_LINDEN:2014}, when for any $f:A\to B$ and $g:B\to C$ in $\C$,
if $f$ and $g$ are central extensions and $B$ is perfect, then
$gf$ is a central extension.

Our objectives in this paper are to show the following:
\begin{enumerate}[(i)]
\item 
For a semi-abelian category $\C$ 
satisfying \wnh{},
in which centers exist and are characteristic, if $f:X\to Y$ is a morphism
in $\C$ with kernel central, and $\chi_1$ and $\chi_2$
are internal $B$-actions on $X$ such that $f\chi_1=f\chi_2$, then $\chi_1$
and $\chi_2$ restrict to the same $B$-action on $[X,X]$. 
\item For a semi-abelian category $\C$ satisfying \wnh{}, 
$\C$ satisfies \uce{}.
\item\label{iii} For a pointed protomodular category $\C$ satisfying
\wnh{}, if $X$ is an object in $\C$ such that $\comm{1_X}{1_X}$ and the generic split extension with kernel $X$ exists, 
and  $g$ commutes with $c_X$, then
$ig$ commutes with $k\comm{1_X}{1_X}$ (where $i$ and $k$ are as in \eqref{diagram:generic_split}).  
\item Under the same hypothesis as \eqref{iii},
if $X$
is a perfect object in $\C$ such that the generic split extension with kernel $X$ exists, then
(a) the
centralizer of $c_X$ is a zero morphism, (b) the center of $X$ is trivial.
\end{enumerate}
Let us mention that (ii) shows that \sh{} (the so called \emph{Smith is Huq} condition)
can be dropped from the hypothesis of Theorem 3.12 of \cite{GRAY_VAN_DER_LINDEN:2015}, and
 (iii) recovers the fact that if $\theta$ is a central automorphism of a group $X$,
then $\theta([x,y])=[x,y]$ for each $x,y$ in $X$.

\section{Lifting extensions and properties of morphisms with kernel a central monomorphism}
In this section we establish some results about when split extensions can be
lifted and the uniqueness of such liftings, which we will need in the final
section.

We begin by recalling some known facts.
It is well known and easy to prove that for a central extension $f:X\to Y$ 
of groups and elements $a, a', b$ and $b'$ in $X$
if  $f(a)=f(a')$ and $f(b)=f(b')$, then $[a,b]=[a',b']$. Categorically, this
follows from the following fact
(which is essentially the same as Proposition 3.1 of \cite{BOURN_GRAN:2002}).
\begin{proposition}
\label{proposition:monomorphic behavior of central extensions}
Let $\C$ be a pointed protomodular category, $f:A\to C$ a morphism
in $\C$, and $k:K\to A$ the kernel of $f$. If $k$ is central, then
\begin{enumerate}[(a)]
 \item the pair $\varphi,\pi_2 : K\times A \to A$ (where $\varphi: K\times A\to A$ is the
  cooperator of $k$ and $1_A$) is the
 kernel pair of $f$, and the composite $\langle 0,1\rangle \comm{1_A}{1_A}$ is the
  commutator morphism of $1_{K\times A}$ with itself (provided $\comm{1_A}{1_A}$ exists);
\item  For a pair of morphisms $u,v : W\to A$ if $fu=fv$, then
 $u\comm{1_W}{1_W} = v\comm{1_W}{1_W}$ (provided $\comm{1_W}{1_W}$ exists).
\end{enumerate} 
\end{proposition}
\begin{remark}\label{remark:clarification of proposition:monomorphic behavior of central extensions}
Note that in Proposition \ref{proposition:monomorphic behavior of central extensions} (a)
the morphism $\langle 0,1\rangle: A\to K\times A$ is the diagonal of the kernel
pair of $f$. This of course means that for any pullback $(A\times_C A,\pi_1,\pi_2)$ of $f$ along $f$ we have
$\comm{A\times_C A}{A\times_C A} = \langle \comm{1_A}{1_A},\comm{1_A}{1_A} \rangle$.
\end{remark}
We will also need the following slight strengthening of Lemma 2.6 of
\cite{CIGOLI_GRAY_VAN_DER_LINDEN:2015a}.
 \begin{lemma}
 \label{lemma:lifting}
  Let $\C$ be a pointed protomodular category.
  For each split extension as displayed in the bottom row of the diagram
  \[
   \xymatrix{
    X'\ar[d]_{u}\ar@{-->}[r]^{\kappa'}&
    A'\ar@{-->}@<2.15pt>[r]^{\alpha'}\ar@{-->}[d]^{v}&
    B\ar@{-->}@<2.15pt>[l]^{\beta'}\ar@{==}[d]\\
    X\ar[r]_{\kappa}&A\ar@<2.15pt>[r]^{\alpha}&B\ar@<2.15pt>[l]^{\beta}
   }
  \]
  there exists a unique (up to isomorphism) lifting along a monomorphism $u$ to a morphism of split
  extensions, as displayed with dotted arrows,
  as soon as $\kappa u$ is normal (or more generally Bourn-normal, that is,
normal to some equivalence relation in the sense of \cite{BOURN:2000a}).
 \end{lemma}
\begin{proof}
The existence of a lifting follows from Lemma 2.6 of
\cite{CIGOLI_GRAY_VAN_DER_LINDEN:2015a}.
Uniqueness follows from e.g.~Lemma 4.1 of \cite{GRAY:2013b}.
\end{proof}

It is easy to check and well known that for a central extension $f:X\to Y$
of a group $Y$ if $\cdot, *$ are actions of a group $B$ on $X$ 
such that for each $x\in X$ and $b\in B$ we have that
$f(b\cdot x) = f(b*x)$, then both actions have the same restriction to
$\como{X}{X}$. We will see that this can be recovered from
the following categorical fact (which is closely related to
Lemma 3.10 of \cite{GRAY_VAN_DER_LINDEN:2015}):
\begin{proposition}\label{proposition_similtanius_lifting}
Let $\C$ be a pointed protomodular category satisfying \wnh.
For a morphism $f: X\to Z$ with kernel a central monomorphism and
for morphisms of split extensions displayed in the diagram
 \begin{equation}
  \vcenter{
   \label{diagram:morphisms_of_split_exts_with_the_same_morphism_at_kernel}
  \xymatrix{
  X
  \ar[r]^{\kappa_1}
  \ar[d]_{f}
  &
  A_1
  \ar@<2.00pt>[r]^{\alpha_1}
  \ar[d]_{g_1}
  &
  B\ar@<2.00pt>[l]^{\beta_1}
  \ar@{=}[d]\\
  Z
  \ar[r]_{\sigma}
  &
  C
  \ar@<2.00pt>[r]^{\gamma}
  &
  B
  \ar@<2.00pt>[l]^{\delta}
  \\
  X
  \ar[r]_{\kappa_2}
  \ar[u]^{f}
  &
  A_2
  \ar[u]^{g_2}
  \ar@<2.00pt>[r]^{\alpha_2}
  &
  B\ar@{=}[u]
  \ar@<2.00pt>[l]^{\beta_2}
 }
}
 \end{equation}
if $\comm{1_X}{1_X}$ exists, then the unique (up to isomorphism) liftings of the split
extensions at the top and bottom of
\eqref{diagram:morphisms_of_split_exts_with_the_same_morphism_at_kernel}
along $\comm{1_X}{1_X}$ are isomorphic.
\end{proposition}
\begin{proof}
 We begin by forming the pullback
 \begin{equation}
  \vcenter{
   \label{diagram:pullback}
\xymatrix@!@C=-1ex@R=-2ex{
 X\times_Z X
  \ar[dd]_{\pi_1}
  \ar[rr]^-{{\kappa_1\times \kappa_2}}
  \ar[rd]^{\pi_2}&&
{A_1\times_C A_2}
 \ar[dd]^(0.30){{\pi_1}}|\hole
\ar@<2.00pt>[rr]^{{\alpha_1\pi_1}}
\ar[rd]^{{\pi_2}}&&
B
 \ar@{=}[dd]|(0.48)\hole|(0.51)\hole
 \ar@<2.00pt>[ll]^(0.4){{\langle \beta_1,\beta_2\rangle}}\ar@{=}[rd]&\\
&X\ar[rr]^(0.30){\kappa_2}\ar[dd]^(0.25){f}&&
A_2\ar[dd]^(0.3){g_2}\ar@<2.00pt>[rr]^(0.30){\alpha_2}&&
B\ar@{=}[dd]\ar@<2.00pt>[ll]^(0.70){\beta_2}\\
X\ar[rr]^(0.30){\kappa_1}|\hole\ar[rd]_{f}&&
A_1\ar@<2.00pt>[rr]^(0.30){\alpha_1}|\hole\ar[rd]_{g_1}&&
B\ar@<2.00pt>[ll]|\hole^(0.70){\beta_1}\ar@{=}[rd]&\\
&Z
\ar[rr]_{\sigma}&&
C
\ar@<2.00pt>[rr]^{\gamma}&&
B
\ar@<2.00pt>[ll]^{\delta}
}
 }
\end{equation}
in $\SplExt(\C)$.
Since by Proposition
 \ref{proposition:monomorphic behavior of central extensions},
 $\langle \comm{1_X}{1_X},\comm{1_X}{1_X}\rangle$ is the 
 commutator morphism of $1_{X\times_Z X}$ and itself, it follows
 from \wnh{} that
 $(\kappa_1\times\kappa_2) \langle \comm{1_X}{1_X},\comm{1_X}{1_X}\rangle$
 is normal. Therefore, Lemma \ref{lemma:lifting} produces a morphism
 \[
  \xymatrix{
 \como{X}{X}
 \ar[d]_{\langle\comm{1_X}{1_X},\comm{1_X}{1_X}\rangle}\ar[r]^-{{\kappa}}
 &
{A}\ar[d]^{{u_1}}
 \ar@<2.00pt>[r]^{{\alpha}}
 &
B
 \ar@{=}[d]
\ar@<2.00pt>[l]^{{\beta}}
 \\
 X\times_Z X
  \ar[r]_-{{\kappa_1\times \kappa_2}}
  &
{A_1\times_C A_2}
 \ar@<2.00pt>[r]^-{\alpha_1\pi_1}&
B,
 \ar@<2.00pt>[l]^-{\langle \beta_1,\beta_2\rangle}
}
 \]
 which after composing with the morphisms forming the upper and back faces
 of
 \eqref{diagram:pullback},
 produce liftings of the split extensions at the top
 and bottom of
 \eqref{diagram:morphisms_of_split_exts_with_the_same_morphism_at_kernel}
with the same domain, as desired. 
\end{proof}
Translating Proposition
\ref{proposition_similtanius_lifting}
via the equivalence between split extensions and internal actions recalled in the introduction
we obtain:
\begin{proposition}\label{proposition:actions}
Let $\C$ be a semi-abelian category satisfying \wnh. If $f : X\to Z$ is a morphism
in $\C$ with kernel central, and $\chi_1,\chi_2: B\flat X \to X$ and $\zeta: B\flat Z\to Z$
are internal $B$-actions such that $f$ is an internal $B$-action morphism from both
$(X,\chi_1)$ and $(X,\chi_2)$ to $(Z,\zeta)$, then 
$\chi_1$ and $\chi_2$ restrict to the same internal $B$-action on $[X,X]$.
\end{proposition}
The following example shows that the assumption \wnh{} cannot be omitted from the previous
proposition.
\begin{example}
We consider the category of non-associative algebras over a commutative ring $R$.
Recall that this is a 
semi-abelian category which does not satisfy \wnh{} (\cite{CIGOLI:2010}). Recall also that
if $S$ and $T$ are sub-non-associative-algebras of $X$, then $S$ and $T$ commute if and
only if $st=0=ts$ for each $s$ in $S$ and $t$ in $T$. The center of $X$ is
$Z(X)=\{x\in X\mid \forall y\in X, xy=0=yx\}$ and the
commutator $[X,X]$ is the sub-non-associative-algebra generated as an $R$-module by 
$\{xy\mid x\in X,y\in X\}$.
One can show that an internal $B$-action $X$ in the sense of \cite{BOURN_JANELIDZE:1998} is the
same as a non-associative algebra $X$ together with two $R$-bilinear
maps $l,r:B\times X\to X$. Let $X$, $Z$ and $B$ be the free $R$-modules on $\{x,y,z\}$, 
$\{y,z\}$ and $\{b\}$, respectively, with
multiplication defined on generators by\\
\begin{center}
\begin{tabular}{r|rrr}
 & x & y & z\\
 \hline
x & 0 & 0 & 0 \\
y & 0 & x & x\\
z & 0 & x & x
\end{tabular}
\hspace{1cm}
\begin{tabular}{r|rr}
 & y & z\\
 \hline
y & 0 & 0\\
z & 0 & 0
\end{tabular}
\hspace{1cm}
\begin{tabular}{r|r}
 & b \\
 \hline
b & b 
\end{tabular}
,\ \end{center}
respectively. Let $f:X\to Z$ be the homomorphism sending $f(x)=0$, $f(y)=y$ and $f(z)=z$.
Note that $\textnormal{Ker}(f)=Z(X)=[X,X]$ and
is generated by $x$ as a sub-$R$-module. We can make $X$ into a $B$-action in two different ways
by defining $l(b,u)=x=r(b,u)$, and $l(b,u)=0=r(b,u)$ for $u\in\{x,y,z\}$, respectively. 
If we equip $Z$ with the action defined by $l(b,v)=0=r(b,v)$ for $v\in \{y,z\}$, then 
$f$ is a $B$-action morphism in both cases.
It is immediate that these two $B$-actions on $X$
restrict to distinct $B$-actions on $[X,X]$.
\end{example}
Since the center of every group exists
and is characteristic, we can recover the
group-theoretical fact mentioned just above Proposition
\ref{proposition_similtanius_lifting} as a special case of the following
proposition. 
\begin{proposition} \label{proposition:actions_2}
Let $\C$ be a semi-abelian category satisfying \wnh{} in which the center of
each object exists and is characteristic. 
If $f : X\to Z$ is a morphism
in $\C$ with kernel central and $\chi_1,\chi_2: B\flat X \to X$ 
are internal $B$-actions such that $f\chi_1=f\chi_2$, then 
$\chi_1$ and $\chi_2$ restrict to the same internal  $B$-action on $[X,X]$.
\end{proposition}
\begin{proof}
Without loss of generality we may assume that $f$ is a normal
epimorphism (since otherwise factorizing $f$ as $me$ where
$m$ is a monomorphism and $e$ is a normal epimorphism produces a normal
epimorphism
$e$ with the same properties as $f$). Since the kernel of $f$ is central
it follows that it factors through $z_X$, the center of $X$, and
hence, $q:X\to Y$, the cokernel of $z_X$, factors through $f$.
Since $z_X$ is characteristic, it follows by Proposition 2.7 of
\cite{CIGOLI_MONTOLI:2015} that there is a unique morphism $\zeta : B\flat Y\to Y$
such that $q$ is an internal $B$-action morphism from $(X,\chi_1)$ to $(Y,\zeta)$.
Since $f\chi_1=f\chi_2$, it follows that $q\chi_1=q\chi_2$ and so $q$ is also a morphism
from $(X,\chi_2)$ to $(Y,\zeta)$. We conclude via Proposition
\ref{proposition:actions}. 
\end{proof}
\begin{remark}
It seems worth mentioning that if a category $\C$ has centralizers of
normal monomorphisms which are normal, then centers exist and are
characteristic. In such a category, one easily proves that if $k:X\to Y$ is
normal monomorphism, then $k\wedge z_k \cong kz_X$. 
Proposition 5.2 and Corollary 3.5 of
\cite{BOURN_JANELIDZE:2009}  show that centralizers of normal monomorphisms in action accessible categories
are normal. This means that examples of such
categories include all categories of interest in the sense of Orzech
\cite{ORZECH:1972}, see
Theorem 4.7 of \cite{MONTOLI:2010} for action accessibility and Theorem 5.3.6 of \cite{CIGOLI:2010} (which appeared later as Theorem 4.4 of \cite{CIGOLI_GRAY_VAN_DER_LINDEN:2015a}) for \nh{} .
In particular the categories of groups,
 rings (\emph{without 1}), as well as 
of commutative, associative, Boolean and Lie algebras
over an arbitrary commutative ring, are all examples.
\end{remark}
As another concrete example, let us explain what Proposition \ref{proposition:actions_2}
means for the category $\CAlg_R$ of commutative algebras over a commutative ring $R$.
\begin{example}
For $B$ and $X$ in $\CAlg_R$, an action of $B$ on $X$ in the sense of \cite{BOURN_JANELIDZE:1998} turns out
to be equivalent to $X$ being equipped with an action of $B$ on $X$ making $X$
into a commutative $B$-algebra \cite{BORCEUX_JANELIDZE_KELLY:2005b}. For subobjects $S,T$ of a commutative algebra $A$, $S$ and $T$ commute
if and only if for each $s$ in $S$ and $t$ in $T$, $st=0$. If
$S$ and $T$ are normal (ideals) then $[S,T]=ST$ is the commutative algebra generated as
an $R$-module by $\{st\mid s\in S, t\in T\}$.
Since $\CAlg_R$ is a category of interest it satisfies \wnh{}, see \cite{CIGOLI:2010}.
The centralizer of $S$ in $A$ exists and is defined by
$\{a\in A\mid \forall s \in S, sa=0\}$. If $S$ is normal, then
so is its centralizer.
Proposition \ref{proposition:actions_2} tells us that if $f:X\to Z$ has kernel central
and $X$ admits two different actions of $B$ (say $*$ and $\cdot$) such that
$f(b*x)=f(b\cdot x)$, then the restriction of the two actions $*$ and $\cdot$ to $[X,X]$ are equal.
\end{example}
We end this section by showing that \uce{} holds in a semi-abelian category
satisfying \wnh{}. Looking at the proof of Theorem 3.12 from \cite{GRAY_VAN_DER_LINDEN:2015}
we see that it is sufficient to prove Lemma 3.11 under the weaker
assumption of \wnh{}. In fact the following stronger result holds:
\begin{lemma}
Let $\C$ be a semi-abelian category satisfying \wnh{} and let $f:A\to B$ and $g:B\to C$
be central extensions in $\C$. The kernel of $gf$ 
commutes with $\comm{1_A}{1_A}:[A,A]\to A$.
\end{lemma}
\begin{proof}
Let $k:K\to A$ be the kernel of $gf$. Since $g(fk)=0$ and $g$ is a central extension,
it follows that $fk$ commutes with $1_B$, and hence $fk$ commutes with $f$. Consequently, the diagram
\[
  \xymatrix{
  A
  \ar[r]^{\langle 1,0\rangle}
  \ar[d]_{f}
  &
  A\times K
  \ar@<2.00pt>[r]^{\pi_2}
  \ar[d]_{\langle \varphi,1\rangle}
  &
  K\ar@<2.00pt>[l]^{\langle 0,1\rangle}
  \ar@{=}[d]\\
  B
  \ar[r]_{\langle 1,0\rangle}
  &
  B\times K
  \ar@<2.00pt>[r]^{\pi_2}
  &
  K
  \ar@<2.00pt>[l]^{\langle fk,1\rangle}
  \\
  A
  \ar[r]_{\langle 1,0\rangle}
  \ar[u]^{f}
  &
  A\times K
  \ar[u]^{f\times 1}
  \ar@<2.00pt>[r]^{\pi_2}
  &
  K\ar@{=}[u]
  \ar@<2.00pt>[l]^{\langle k,1\rangle}
 }
\]
in which $\varphi$ is the cooperator of $f$ and $fk$, is a cospan of split extensions.
According to Proposition \ref{proposition_similtanius_lifting}, the split extensions at the top and bottom have
isomorphic lifting along $\comm{1_A}{1_A}$. Since  
\[
  \xymatrix{
  [A,A]
  \ar[r]^{\langle 1,0\rangle}
  \ar[d]_{\comm{1_A}{1_A}}
  &
  [A,A]\times K
  \ar@<2.00pt>[r]^{\pi_2}
  \ar[d]_{\comm{1_A}{1_A}\times 1}
  &
  K\ar@<2.00pt>[l]^{\langle 0,1\rangle}
  \ar@{=}[d]\\
  A
  \ar[r]_{\langle 1,0\rangle}
  &
  A\times K
  \ar@<2.00pt>[r]^{\pi_2}
  &
  K
  \ar@<2.00pt>[l]^{\langle 0,1\rangle}
 }
\]
is a lifting of the upper split extension along $\comm{1_A}{1_A}$, it follows
that there is a morphism $\theta : [A,A]\times K \to A\times K$ making the diagram
\[
  \xymatrix{
  [A,A]
  \ar[r]^{\langle 1,0\rangle}
  \ar[d]_{\comm{1_A}{1_A}}
  &
  [A,A]\times K
  \ar@<2.00pt>[r]^{\pi_2}
  \ar[d]_{\theta}
  &
  K\ar@<2.00pt>[l]^{\langle 0,1\rangle}
  \ar@{=}[d]\\
  B
  \ar[r]_{\langle 1,0\rangle}
  &
  B\times K
  \ar@<2.00pt>[r]^{\pi_2}
  &
  K
  \ar@<2.00pt>[l]^{\langle k,1\rangle}
 }
\]
into a lifting. It follows that $\pi_1\theta$ is the cooperator of $\comm{1_A}{1_A}$ of $k$,
and hence they commute.
\end{proof}
Copying the argument from \cite{GRAY_VAN_DER_LINDEN:2015} but replacing Lemma 3.11 by the previous
lemma we obtain:
\begin{theorem}
If a semi-abelian category satisfies \wnh{}, then it satisfies \uce{}.\qed
\end{theorem}
\section{Split extension classifiers of perfect objects}
In this section, we apply the preceding results
to generalize the result that every central automorphism
of a perfect group is trivial.

Recall that a reflexive graph $(A,B,\alpha_1,\alpha_2,\beta)$ in a category
$\C$ is a quadruple where $A$ and $B$ are objects in $\C$, and
$\alpha_1, \alpha_2: A\to B$ and $\beta : B\to A$ are morphisms in $\C$
such that $\alpha_1\beta = 1_B=\alpha_2\beta$. When $\C$ is protomodular
(more generally \emph{Mal'tsev} \cite{CARBONI_LAMBECK_PEDICCHIO:1991}) a reflexive relation $(A,B,\alpha_1,\alpha_2,\beta)$
admits a multiplication if and only if the kernel pairs of $\alpha_1$ and
$\alpha_2$ \emph{Smith-commute} \cite{CARBONI_PEDICCHIO_PIRAVANO:1992}. When this is the case it follows
that the kernels of $\alpha_1$ and $\alpha_2$ Huq-commute \cite{BOURN_GRAN:2002a}, which, in turn, implies that,
if $\kappa$ is the kernel of $\alpha_1$, then $\alpha_2 \kappa$ has kernel central.
We have:
\begin{theorem}
\label{main_theorem}
 Let $\C$ be a pointed protomodular category satisfying \wnh{},
  let $(A,B,\alpha_1,\alpha_2,\beta)$ be a reflexive graph in $\C$, and let
$\kappa:X\to A$ be the kernel of $\alpha_1$, such that $\alpha_2\kappa$ has kernel central.  
\begin{enumerate}[(a)]
\item \label{main_1} If a morphism $g:Y\to B$ commutes
with $\alpha_2 \kappa$ and
$\comm{1_X}{1_X}$ exists, then
 $\kappa \comm{1_X}{1_X}$ and $\beta g$ commute.
\item \label{main_1.5} If $X$ is perfect and a morphism $g:Y\to B$ commutes
with $\alpha_2 \kappa$, then
 $\kappa$ and $\beta g$ commute.
\item \label{main_2}If $X$ is perfect and the split extension 
\[
\xymatrix{
X \ar[r]^{\kappa} & A\ar@<0.5ex>[r]^{\alpha_1} & B\ar@<0.5ex>[l]^{\beta}
}
\]
is faithful, then the centralizer of $\alpha_2 \kappa$ is trivial.
\end{enumerate}
\end{theorem}

\begin{proof}
Suppose $\comm{1_X}{1_X}$ exists and that $g:Y\to B$ is a morphism commuting with $f=\alpha_2\kappa$. Consider the diagram
\begin{equation}
\label{big_diagram}
\vcenter{
\xymatrix@!@C=-4ex@R=-4ex{
 &&\como{X}{X}
 \ar[ld]_{\comm{1_X}{1_X}}
 \ar[ddd]^(0.18){\comm{1_X}{1_X}}|(0.33){\hole}|(0.67){\hole}
\ar[rrr]^-{\langle 0,1\rangle}
&&&
 Y\times \como{X}{X}
\ar@<0.5ex>[rrr]^{\pi_1}
\ar@{-->}[ld]_{u}
 \ar[ddd]^(0.18){1\times \comm{1_X}{1_X}}
|(0.3){\hole}
|(0.36){\hole}
|(0.64){\hole}
|(0.7){\hole}
&&&
Y
\ar@{=}[ld]
\ar@{=}[ddd]
\ar@<0.5ex>[lll]^{\langle 1,0\rangle}
\\
&X
\ar@{=}[ld]
\ar[rrr]^{\tilde \kappa}
\ar[ddd]^(0.18){f}|(0.33){\hole}
&&&
\tilde A
\ar@<0.5ex>[rrr]^(0.67){\tilde \alpha}
\ar[ddd]^(0.16){\langle\tilde \alpha,\alpha_2 h\rangle}|(0.3){\hole}|(0.36){\hole}
\ar[ld]_{h}
&&&
Y
\ar[ld]^(0.6){g}
\ar@{=}[ddd]
\ar@<0.5ex>[lll]^(0.33){\tilde \beta}
&
\\
X
\ar[ddd]_{f}
\ar[rrr]^{\kappa}
&&&
A
\ar@<0.5ex>[rrr]^{\alpha_1}
 \ar[ddd]_(0.45){\langle \alpha_1,\alpha_2\rangle}
&&&
B
\ar@{=}[ddd]
\ar@<0.5ex>[lll]^{\beta}
&&
\\
&&
X
\ar[ld]^{f}
\ar[rrr]|(0.33){\hole}^{\langle 0,1\rangle}|(0.67){\hole}
&&&
Y\times X
\ar@<0.5ex>[rrr]|(0.33){\hole}^{\pi_1}|(0.67){\hole}
\ar[ld]^(0.4){\langle \pi_1,\varphi\rangle}
&&&
Y
\ar@{=}[ld]
\ar@<0.5ex>[lll]|(0.33){\hole}^{\langle 1,0\rangle}|(0.67){\hole}
\\
&
B
\ar@{=}[ld]
\ar[rrr]^(0.45){\langle 0,1\rangle}|(0.67){\hole}
&&&
Y\times B
\ar@<0.5ex>[rrr]^{\pi_1}|(0.67){\hole}
\ar[ld]^(0.3){g\times 1}
&&&
Y
\ar[ld]^{g}
\ar@<0.5ex>[lll]|(0.33){\hole}^{\langle 1,g\rangle}
&
\\
B
\ar[rrr]^{\langle 0,1\rangle}
&&&
B\times B
\ar@<0.5ex>[rrr]^{\pi_1}
&&&
B
\ar@<0.5ex>[lll]^{\langle 1,1\rangle}
}
}
\end{equation}
in which
\begin{description}
 \item{-} $\varphi$ is the cooperator of $f$ and $g$;
\item{-} the object $(\tilde A,h,\tilde \alpha)$ is the pullback of
 $\alpha_1$ and $g$;
\item{-} the morphisms $\tilde \kappa$ and $\tilde \beta$ are the unique
 morphisms such that $h\tilde \kappa = \kappa$, $\tilde \alpha \tilde \kappa=0$,
  $h\tilde \beta = \beta g$
  and $\tilde \alpha \tilde \beta=1_Y$.
\end{description}
According to Proposition \ref{proposition_similtanius_lifting} (via
Lemma \ref{lemma:lifting}) there must be a
morphism $u:Y\times [X,X]\to \tilde A$ such that Diagram \eqref{big_diagram} consists of morphisms of split extensions.
The proof of \eqref{main_1} is completed by noting that $hu$ produces
a cooperator for $\kappa \comm{1_X}{1_X}$ and $\beta g$.
To prove \eqref{main_1.5}, we note that \eqref{main_1} tells us that
for any $g$ commuting with $f$, the morphisms $\kappa \comm{1_X}{1_X}$ and $\beta g$ commute.
Therefore, when $X$ is perfect, since $\comm{1_X}{1_X}$ is an isomorphism, it follows that $\kappa$ and $\beta g$ commute. 
To prove \eqref{main_2}, just recall that Proposition 5.2 of
\cite{BOURN_JANELIDZE:2009} (see also  Lemma 2.4 of \cite{CIGOLI_MANTOVANI:2012})
tells us that any morphism $g$ such that $\kappa$ and $\beta g$ commute,
must be a zero morphism.
\end{proof}
\begin{remark}  
Recall that a split extension \eqref{diagram:split} can be equivalently defined
to be eccentric \cite{BOURN:2013b}, if each morphism $g:Y\to B$ such that
$\beta g$ and $\kappa$ commute, is a zero morphism. Note that, since the above
mentioned, Proposition 5.2 of \cite{BOURN_JANELIDZE:2009} which show, using this terminology,
that ``faithful'' implies ``eccentric'', one can generalize part (b) of
Theorem \ref{main_theorem} by replacing
``faithful'' with ``eccentric''.
\end{remark}
Suppose that $X$ is an object in a pointed protomodular category,
such that $\spexc{X}$ exists.
Since the generic split extension \eqref{diagram:generic_split} is
necessarily faithful and (as recalled above) forms part of a
groupoid with underlying reflexive graph $(\spexc{X}\ltimes X, p_1,p_2, i)$
such that $c_X=p_2k$,
the previous theorem produces the following corollaries.
\begin{corollary}\label{corollary:a}
Let $\C$ be a pointed protomodular category satisfying \wnh{}.  If $X$ is an
object in $\C$ such that $\comm{1_X}{1_X}$,
$\spexc{X}$ and $z_{c_X}$
exist, then the morphisms $k \comm{1_X}{1_X}$ and $i z_{c_X}$ commute.
\end{corollary}
\begin{corollary}\label{corollary:b}
Let $\C$ be a pointed protomodular category satisfying \wnh{}.
If $X$ is a perfect object in $\C$ such that $\spexc{X}$ exists,
 then $z_{c_X}=0$ and hence $z_{\spexc{X}}=0$.
\end{corollary}

\begin{remark}
Corollary \ref{corollary:a} applied to the category of groups
shows that if  $\theta : X\to X$ is a central automorphism, then
$\theta([x,y]) = [x,y]$.
Applied to the category of Lie algebras it shows that if $d$ is a central
derivation, then $d([x,y])=0$.
\end{remark}
As an immediate consequence of Corollary \ref{corollary:b} we obtain:
\begin{corollary}\label{corollary:perf_action_rep}
Let $\C$ be a pointed protomodular category admitting all generic split
 extensions and satisfying \wnh{}.
If $X$ is a perfect object in $\C$, then $z_{c_X}=0$ and hence $z_{\spexc{X}}=0$.
\end{corollary}
\begin{example}\label{example:semi-abelian examples}
Recall from \cite{CIGOLI_GRAY_VAN_DER_LINDEN:2015b} that \emph{algebraic
coherence} implies \nh{}, which is stronger than \wnh{}. Recall also that there
are plenty of algebraically coherent categories, including
all categories of interest in the sense of G. Orzech \cite{ORZECH:1972}.
Consequently, the above results apply to all action representable
algebraically coherent categories. These include
the categories of groups, Lie algebras, Boolean rings and cocommutative Hopf
algebras (for action representability of the first three categories see
\cite{BORCEUX_JANELIDZE_KELLY:2005b} and for the last one see 
\cite{GRAN_STERCK_VERCRUYSSE:2019},
for algebraic coherence see 
\cite{CIGOLI_GRAY_VAN_DER_LINDEN:2015b}).
Further examples can be obtained by taking a functor category of, or
the category of internal groupoids in,
a semi-abelian, algebraically coherent category with split extension
classifiers and normalizers (see \cite{GRAY:2013b} and \cite{GRAN_GRAY:2021}).
This produces, among many others,
crossed modules of groups and Lie algebras as further examples.
\end{example}
\begin{proposition}
If $\mathcal{V}$ is a variety of semi-abelian algebras satisfying \wnh{}, then 
the category $\mathcal{V}(\mathbf{Top})$ of internal $\mathcal{V}$ algebras in
the category of topological spaces satisfies \wnh{}.
\end{proposition}
\begin{proof}
Recall that coequalizers and equalizers in $\mathcal{V}(\mathbf{Top})$ are calculated as in $\mathcal{V}$, equipped with
the quotient and subspace topology, respectively (see 
\cite{BORCEUX_CLEMENTINO:2005}). This means that
$\comm{1_X}{1_X}= \ke(\coke(\langle 1_X,1_X\rangle) \langle 1,0\rangle)$
is calculated as in $\mathcal{V}$ with the subspace topology. Since a subspace 
of a subspace is a subspace of the whole space, this also implies
that a normal monomorphism in $\mathcal{V}(\mathbf{Top})$ is characteristic
if and only if it is characteristic in $\mathcal{V}$.
\end{proof}
\begin{example}\label{example:top groups}
The category of topological groups has been shown to be action
representable \cite{CAGLIARI_CLEMENTINO:2019} and satisfies \wnh{} by the previous proposition.
It follows that our results hold for the category of topological groups.
\end{example}
\providecommand{\bysame}{\leavevmode\hbox to3em{\hrulefill}\thinspace}
\providecommand{\MR}{\relax\ifhmode\unskip\space\fi MR }
\providecommand{\MRhref}[2]{%
  \href{http://www.ams.org/mathscinet-getitem?mr=#1}{#2}
}
\providecommand{\href}[2]{#2}

\end{document}